\documentclass[10pt,a4paper,reqno]{amsart} 
\usepackage[final]{optional}
\usepackage[british,american]{babel}
\usepackage[margin=0cm]{geometry}
\usepackage{lipsum}
\usepackage[leqno]{amsmath}
\usepackage{mathrsfs}

\usepackage{amsfonts, amsmath, wasysym, caption}

\usepackage{amssymb,amsthm, paralist}
\usepackage{bbm}

\usepackage[normalem]{ulem}
\usepackage{bbold}

\usepackage{
latexsym,
}

\usepackage[usenames]{color}
\usepackage{tikz}
\usetikzlibrary{decorations.pathreplacing}
\usepackage{url}

\definecolor{darkgreen}{rgb}{0,0.5,0}
\definecolor{darkred}{rgb}{0.7,0,0}

\usepackage[colorlinks, 
citecolor=darkgreen, linkcolor=darkred
]{hyperref}

\usepackage{esint}
\usepackage[normalem]{ulem}

\usepackage{enumitem}
\theoremstyle{plain}
\newtheorem{lemma}{Lemma}[section]
\newtheorem{theorem}[lemma]{Theorem}
\newtheorem{proposition}[lemma]{Proposition}
\theoremstyle{definition}
\newtheorem{definition}[lemma]{Definition}

\numberwithin{equation}{section}

\newcommand{\ga}{\gamma}

\newcommand{\de}{\delta}
\newcommand{\De}{\Delta}
\newcommand{\om}{\omega}

\newcommand{\ka}{\kappa}

\renewcommand{\th}{\theta}

\newcommand{\R}{\ensuremath{{\mathbb R}}}

\let\div\relax
\DeclareMathOperator{\div}{div}

\newcommand{\beq}{\begin{equation}}
\newcommand{\eeq}{\end{equation}}
\newcommand{\beqs}{\begin{equation*}}
\newcommand{\eeqs}{\end{equation*}}
\newcommand{\beqa}{\begin{equation}\begin{aligned}}
\newcommand{\eeqa}{\end{aligned}\end{equation}}
\newcommand{\beqas}{\begin{equation*}\begin{aligned}}
\newcommand{\eeqas}{\end{aligned}\end{equation*}}

\newcommand{\half}{\frac{1}{2}}

\newcommand{\eps}{\varepsilon}

\newcommand{\supp}{\text{supp}} 

\title[Spherically Symmetric Solutions of the Euler Equations]{{\sc
Spherically Symmetric Solutions of the Multi-dimensional, Compressible, Isentropic Euler Equations
}
\\ 
}
\author[Matthew R. I. Schrecker]{Matthew R. I. Schrecker}\thanks{University of Wisconsin Madison, Van Vleck Hall, 480 Lincoln Drive, Madison, Wisconsin 53706, USA\\
\indent Email: schrecker@wisc.edu}

\begin{document}
\begin{abstract}
In this paper, we prove the existence of finite-energy weak solutions to the compressible, isentropic Euler equations given arbitrary spherically symmetric initial data of finite energy. In particular, we show that the solutions obtained to the spherically symmetric Euler equations in the recent works by Chen-Perepelitsa and Chen-Schrecker, \cite{ChenPerep2, ChenSchrecker}, are weak solutions of the multi-dimensional compressible Euler equations. This follows from new uniform estimates made on artificial viscosity approximations up to the origin, removing previous restrictions on admissible test functions and ruling out formation of an artificial boundary layer at the origin. The uniform estimates may be of independent interest concerning the possible rate of blow-up of density and velocity at the origin for spherically symmetric flows.
\end{abstract}

\maketitle

\section{Introduction and Main Result}
The spherically symmetric, isentropic Euler equations have been a subject of active interest since at least the 1940s. In several pioneering works, especially those of Guderley \cite{Guderley}, \textit{cf}.~Courant and Friedrichs \cite{CourantFriedrichs}, certain special solutions were analysed, giving evidence of the possibility of finite-time blow-up of the density and velocity at the origin for  spherically symmetric solutions (see also the recent work of Jenssen and Tsikkou \cite{Jenssen} for the full Euler system). However, the general question of existence of spherically symmetric solutions of the compressible, isentropic Euler equations for arbitrary spherically symmetric initial data has remained open until now, except for the case excluding the origin, solved by Chen \cite{Chen}. The compressible, isentropic Euler equations in $\R^n$ are
\beq\label{eq:Euler}
\begin{cases}
\partial_t\rho+\div_{\mathbf{x}}(\rho\mathbf{u})=0, &(t,{\mathbf{x}})\in\R_+\times\R^n,\\
\partial_t(\rho\mathbf{u})+\div_{\mathbf{x}}\big(\rho\mathbf{u}\otimes\mathbf{u}\big)+\nabla_{\mathbf{x}} p(\rho)=0, &(t,{\mathbf{x}})\in\R_+\times\R^n,
\end{cases}
\eeq
where $\rho:\R_+\times\R^n\to\R$ is the density of a given fluid (and hence $\rho\geq 0$), $\mathbf{u}:\R_+\times\R^n\to\R^n$ is its velocity, and the scalar function $p(\rho)\geq 0$ is the pressure. We write $\R_+=(0,\infty)$ throughout. In this work, we will consider the pressure laws given by the equation of state of a polytropic gas, that is $p(\rho)=\ka\rho^\ga$ for some $\ga\in(1,\infty)$ and $\ka>0$. By appropriate scaling, we assume without loss of generality that $\ka=(\ga-1)^2/4\ga$.

We consider the Cauchy problem for \eqref{eq:Euler} by imposing initial data
\beq\label{eq:multiDCauchydata}
(\rho,\mathbf{u})|_{t=0}=(\rho_0,\mathbf{u}_0).
\eeq
We recall that a pair $(\rho,\mathbf{u})$ is said to be of finite energy for the Euler equations if
$$E_*[\rho,\mathbf{u}]:=\int_{\R^n}\big(\half\rho|\mathbf{u}|^2+\frac{\ka\rho^\ga}{\ga-1}\big)\,d{\mathbf{x}}<\infty.$$
\begin{definition}\label{def:multiDsoln}
Let $(\rho_0,\mathbf{u}_0)\in L^1_{loc}(\R^n;\R^{n+1})$ be of finite energy, $\rho_0\geq 0$. We say a pair of functions $(\rho,\mathbf{u})\in L^1_{loc}(\R_+\times\R^n;\R^{n+1})$ with $\rho\geq 0$ is a \textit{finite-energy weak solution} of \eqref{eq:Euler}--\eqref{eq:multiDCauchydata} if $E_*[\rho,\mathbf{u}](t)<\infty$ for almost every $t>0$ and, for all $\varphi\in C_c^\infty([0,\infty)\times\R^n)$,
\beqs
\int_0^\infty\int_{\R^n}\big(\rho\varphi_t+\rho\mathbf{u}\cdot \nabla_{\mathbf{x}}\varphi\big)\,d{\mathbf{x}}\,dt+\int_{\R^n}\rho_0({\mathbf{x}})\varphi(0,{\mathbf{x}})\,d{\mathbf{x}}=0,
\eeqs
and, for all $\mathbf{\varphi}\in C_c^\infty([0,\infty)\times\R^n;\R^n)$,
\beqs
\int_0^\infty\int_{\R^n}\big(\rho\mathbf{u}\cdot\mathbf{\varphi}_t+(\rho\mathbf{u}\otimes\mathbf{u}): \nabla_{\mathbf{x}}\mathbf{\varphi}+p(\rho)\div_{\mathbf{x}}\mathbf{\varphi}\big)\,d{\mathbf{x}}\,dt+\int_{\R^n}\rho_0\mathbf{u}_0(\mathbf{x})\cdot\mathbf{\varphi}(0,{\mathbf{x}})\,d{\mathbf{x}}=0.
\eeqs
\end{definition}
For spherically symmetric motion, there exist scalar functions $\rho(t,r)$ and $u(t,r)$, where $r=|\mathbf{x}|$, such that
\beq\label{eq:sphericalconversion}
\rho(t,{\mathbf{x}})=\rho(t,r),\qquad \mathbf{u}(t,{\mathbf{x}})=u(t,r)\frac{{\mathbf{x}}}{|{\mathbf{x}}|}.
\eeq
Then, defining the momentum $m=\rho u$, the Euler equations \eqref{eq:Euler} take the form 
\beq\label{eq:Eulersphsym mform}
\begin{cases}
\big(r^{n-1}\rho\big)_t+\big(r^{n-1}m\big)_r=0, &(t,r)\in\R_+\times\R_+,\\
\big(r^{n-1}m\big)_t+\big(r^{n-1}\frac{m^2}{\rho}\big)_r+r^{n-1}p(\rho)_r=0, &(t,r)\in\R_+\times\R_+.
\end{cases}
\eeq
\begin{definition}\label{def:sphsymsoln}
Let $(\rho_0,m_0)\in L^1_{loc}(\R_+;\R^2)$ with $\rho_0\geq 0$ and $m_0=\rho_0u_0$ be of finite-energy, i.e.
$$E[\rho_0,m_0]:=\int_0^\infty \big(\half\frac{m_0^2}{\rho_0}+\frac{\ka\rho_0^\ga}{\ga-1}\big)r^{n-1}\,dr=E_0<\infty.$$
Then a pair of functions $(\rho,m)\in L^1_{loc}(\R_+^2;\R^2)$ with $\rho\geq0$ is a \textit{finite-energy weak solution} of the spherically symmetric Euler equations \eqref{eq:Eulersphsym mform} with initial data $(\rho_0,m_0)$ if $E[\rho,m](t)<\infty$ for almost every $t>0$ and, for all $\varphi\in C_c^\infty([0,\infty)^2)$,
\beq\label{eq:weakformcty}
\int_0^\infty\int_0^\infty\big(\rho\varphi_t+m\varphi_r\big) r^{n-1}\,dr\,dt+\int_0^\infty \rho_0(r)\varphi(0,r)r^{n-1}\,dr=0,
\eeq
and, for all $\varphi\in C_c^\infty([0,\infty)^2)$ such that $\varphi(t,0)=0$ for all $t\geq 0$,
\beqa\label{eq:weakformmomentum}
\int_0^\infty\int_0^\infty \Big(m\varphi_t+\frac{m^2}{\rho}\varphi_r+p(\rho)\big(\varphi_r+&\frac{n-1}{r}\varphi\big)\Big)r^{n-1}\,dr\,dt\\&+\int_0^\infty m_0(r)\varphi(0,r)r^{n-1}\,dr=0.
\eeqa
\end{definition}
The formulations of Definition \ref{def:multiDsoln} and Definition \ref{def:sphsymsoln} are equivalent via \eqref{eq:sphericalconversion} (see Appendix and e.g.~\cite[Theorem 5.7]{Hoff} for details).
The main result of this note is 
\begin{theorem}\label{theorem}
Suppose $p(\rho)=\ka\rho^\ga$, $\ga>1$. Let $(\rho_0,\mathbf{u}_0)\in L^1_{loc}(\R^n;\R^{n+1})$, $\rho_0\geq 0$, be spherically symmetric data of finite energy. Then there exists a spherically symmetric finite-energy weak solution $(\rho,\mathbf{u})$ of the Euler equations \eqref{eq:Euler}--\eqref{eq:multiDCauchydata} in the sense of Definition \ref{def:multiDsoln}.

In particular, there exist functions $\rho(t,r)$ and $u(t,r)$ such that
\beq
\rho(t,\mathbf{x})=\rho(t,r),\qquad \mathbf{u}(t,\mathbf{x})=u(t,r)\frac{\mathbf{x}}{|\mathbf{x}|},
\eeq
where $(\rho(t,r),m(t,r))$ with $m=\rho u$ is a finite-energy weak solution of the spherically symmetric Euler equations \eqref{eq:Eulersphsym mform} in the sense of Definition \ref{def:sphsymsoln}.
\end{theorem}
In \cite{ChenPerep2}, Chen--Perepelitsa solved system \eqref{eq:Eulersphsym mform} for weak solutions with a restricted weak formulation for $\ga\in(1,3]$ via a vanishing artificial viscosity method, using the following approximate equations for viscosity $\eps>0$ on a truncated domain, $(t,r)\in(0,T)\times(a(\eps),b(\eps))$,
\beq\label{eq:approx}
\begin{cases}
(r^{n-1}\rho^\eps)_t+(r^{n-1}m^\eps)_r=\eps(r^{n-1}\rho^\eps_r)_r,\\
(r^{n-1}m^\eps)_t+(r^{n-1}\frac{(m^\eps)^2}{\rho^\eps})_r+r^{n-1}p_\de(\rho^\eps)_r=\eps(r^{n-1}m^\eps)_{rr}-\eps\frac{n-1}{r}(r^{n-1}m^\eps)_r,
\end{cases}
\eeq
with smooth approximate initial data 
\beq\label{eq:approxic}
(\rho^\eps,m^\eps)|_{t=0}=(\rho_0^\eps,m_0^\eps)
\eeq
 and mixed Dirichlet/Neumann boundary conditions
\beq\label{eq:approxbcs}
(\rho^\eps_r,m^\eps)|_{r=a}=(0,0),\quad (\rho^\eps,m^\eps)|_{r=b}=(\bar\rho(\eps),0),
\eeq
with $\bar\rho(\eps)\rightarrow0$ as $\eps\to0$, where $p_\de(\rho)=p(\rho)+\de\rho^2$ and $\de\to0$ as $\eps\to0$. Here $a(\eps)\in(0,1)$, $b(\eps)\in(1,\infty)$ for each $\eps>0$ and, as $\eps\to 0$, $a(\eps)\to0$, $b(\eps)\to\infty$.
Subsequently, Chen and the author showed in \cite{ChenSchrecker} how the construction could be extended to cover the full range $\ga\in(1,\infty)$.

In the results of \cite{ChenPerep2,ChenSchrecker}, the weak formulation satisfied by the obtained solution $(\rho,m)$ of \eqref{eq:Eulersphsym mform} required restrictions on the space of admissible test functions. In particular, in \cite{ChenPerep2, ChenSchrecker}, it is required that for both equations in \eqref{eq:Eulersphsym mform} the test function $\varphi\in C_c^\infty([0,\infty)^2)$ additionally satisfies $\varphi_r(t,0)=0$ for all $t$ (as well as the correct condition $\varphi(t,0)=0$ for the test function in the momentum equation). Such an assumption restricts the admissible test functions in the weak formulation of \eqref{eq:Euler} (see Appendix for details), and hence it is unclear whether the obtained solutions are indeed weak solutions of \eqref{eq:Euler} in the proper sense of Definition \ref{def:multiDsoln}. In \cite{ChenPerep2,ChenSchrecker}, the additional assumption on test functions at the origin was used primarily to handle the convergence of the flux term in the momentum equation, see \cite[Section 3.4]{ChenPerep2} for details. In particular, the uniform energy bounds (see Proposition \ref{prop:mainenergy}) provide only $L^1$ bounds on the momentum flux up to the origin, hence do not allow for passage to the limit up to the origin.

In this note, we demonstrate that the solutions do indeed satisfy the correct weak formulation by proving uniform estimates on the approximate solutions  up to the origin, $r=0$, allowing for the passage to the limit with general test functions without additional assumptions and the proof of Theorem \ref{theorem}.\footnote{Since the writing of this note, G.-Q. Chen has informed me in a private correspondence that he and Y. Wang have an alternative proof of the full weak formulation, however without the higher integrability estimate up to the origin.}

Both the works \cite{ChenPerep2} and \cite{ChenSchrecker} showed the convergence of the approximate solutions to a strong limit using the technique of compensated compactness (developed by Tartar \cite{Tartar} and Murat \cite{Murat}) in the finite energy framework initiated by LeFloch--Westdickenberg \cite{LeFlochWestdickenberg} for the Euler equations with geometric effects. This framework was subsequently developed by Chen--Perepelitsa in \cite{ChenPerep} and relies crucially on an estimate for the mechanical energy. 

Before stating our new uniform estimates, we therefore first recall from \cite{ChenPerep2, ChenSchrecker} the main energy estimate.

\begin{proposition}\label{prop:mainenergy}
Let $$E_0:=\sup_\eps\int_a^b\big(\half \rho_0^\eps (u_0^\eps)^2+\overline{h_\de}(\rho^\eps_0,\bar\rho)\big)r^{n-1}\,dr<\infty,$$
where $\overline{h_\de}(\rho,\bar\rho)=h_\de(\rho)-h_\de(\bar\rho)-h_\de'(\bar\rho)(\rho-\bar\rho)$ and $h_\de(\rho)=\frac{\ka\rho^\ga}{\ga-1}+\de\rho^2$.

Then, for each $\eps>0$ and any $T>0$, there exists a unique, smooth solution $(\rho^\eps,m^\eps)$ to \eqref{eq:approx}--\eqref{eq:approxbcs} satisfying also
\beqa
\sup_{t\in[0,T]}&\int_a^b\big(\half \rho^\eps (u^\eps)^2+\overline{h_\de}(\rho^\eps,\bar\rho)\big)r^{n-1}\,dr\\
&+\eps\int_0^T\int_a^b\big(h_\de''(\rho^\eps)|\rho^\eps_r|^2+\rho^\eps|u_r^\eps|^2+(n-1)\frac{\rho^\eps (u^\eps)^2}{r^2}\big)r^{n-1}\,dr\,dt\leq E_0.
\eeqa
\end{proposition}
For future use, we note that, for $\bar\rho$ bounded, as $\rho$ grows large, $\overline{h_\de}(\rho,\bar\rho)$ grows as $\rho^\ga$. Hence we have the easy estimate for all $\rho\geq0$,
\beq\label{ineq:star}
\rho+\rho^\ga\leq M(\overline{h_\de}(\rho,\bar\rho)+1).
\eeq
To make our uniform estimates, 
we suppose there exists $M_0>0$, independent of $\eps$, such that
\beq\label{ass:assumptionintro}
\eps\frac{b^n}{a}+\de|\log(a)|\big(1+\frac{b^n}{\eps}\big)+\bar\rho^\th|\log(a)|+\bar\rho^{\ga}b^{n}+\frac{\sqrt{\eps}}{a}\leq M_0.
\eeq
This can always be ensured by careful selection of $\de,\bar\rho,b,a$ depending on $\eps>0$.

The main new uniform estimate that we prove is a higher integrability estimate for both density and velocity. We write $\th=\frac{\ga-1}{2}$, so that $\th>0$ for all $\ga>1$.
\begin{lemma}\label{lemma:highinetgrabilityorigin}
Suppose $(\rho^\eps,m^\eps)$ is a smooth solution of \eqref{eq:approx}--\eqref{eq:approxbcs} on $[0,T]\times(a(\eps),b(\eps))$ with $\inf\rho^\eps=c_\eps>0$ (where $c_\eps$ may depend on $\eps$ and $T$) and that $\eps,\de,a,b,\bar\rho$ satisfy assumption \eqref{ass:assumptionintro}. Let $\om\in C_c^\infty([0,\infty))$ be a test function such that $\om(r)=1$ for $r\in[0,1]$ and $\om(r)\geq0$. Then there exists a constant $M>0$, independent of $\eps$ but depending on $M_0$, such that
\beq
\int_0^T\int_a^b \big(\rho^\eps|u^\eps|^3+(\rho^\eps)^{\ga+\th}\big)\om(r)r^{n-1}\,dr\,dt\leq M.
\eeq
\end{lemma}
This estimate gives us the equi-integrability of the flux term $\big(\rho^\eps (u^\eps)^2+p(\rho^\eps)\big)r^{n-1}$ in system \eqref{eq:approx} all the way up to the origin, $r=0$, and hence allows for the passage to the limit.

The other uniform estimates that we require  concern the spatial derivative of $\rho^\eps$ near the origin, appropriately weighted with the viscosity. These are stated in Lemmas \ref{lemma:rho_r}--\ref{lemma:rho_r/rho} below and are designed to prove the convergence of the viscous terms to zero as $\eps\to0$. 

The structure of this note is as follows. First, in \S\ref{sec:integrability}, we prove Lemma \ref{lemma:highinetgrabilityorigin} using a carefully constructed entropy function and precise estimates around $r=0$. Next, in \S\ref{sec:viscosity}, we give the statements and proofs of Lemmas \ref{lemma:rho_r}--\ref{lemma:rho_r/rho} concerning the spatial derivative of the density. Finally, in \S\ref{sec:mainproof}, we conclude the proof of Theorem \ref{theorem}.

\textbf{Acknowledgement:}
The author would like to thank Gui-Qiang Chen and Helge Kristian Jenssen for useful discussions.

\section{Uniform integrability estimates}\label{sec:integrability}
Throughout this section and \S\ref{sec:viscosity}, we suppose that $(\rho,m)$, $m=\rho u$, is a smooth solution of \eqref{eq:approx}--\eqref{eq:approxbcs} such that $\inf_{(a,b)}\rho(t,r)\geq c_\eps(t)>0$. For simplicity of presentation, we omit the superscript $\eps$ from functions in this section. In order to prove the higher integrability estimate of Lemma \ref{lemma:highinetgrabilityorigin} near the origin, we begin by recalling the weak entropy pair $(\check{\eta},\check{q})$ constructed by Lions, Perthame and Tadmor in \cite[Section I]{LionsPerthameTadmor} by the formulae
\beqas
\check{\eta}(\rho,\rho u)=&\int_\R \half s|s|[\rho^{2\th}-(u-s)^2]_+^{\frac{3-\ga}{2(\ga-1)}}\,ds,\\
 \check{q}(\rho,\rho u)=&\int_\R \half s|s|(\th s+(1-\th) u)[\rho^{2\th}-(u-s)^2]_+^{\frac{3-\ga}{2(\ga-1)}}\,ds.
\eeqas We define a modified entropy pair 
\beqas
&\tilde{\eta}(\rho,m)=\check{\eta}(\rho,m)-\nabla\check{\eta}(\bar\rho,0)\cdot(\rho-\bar\rho,m)\geq 0,\\
&\tilde{q}(\rho,m)=\check{q}(\rho,m)-\nabla\check{\eta}(\bar\rho,0)\cdot\big(m,\frac{m^2}{\rho}+p(\rho)\big).
\eeqas
As shown in \cite{LionsPerthameTadmor, ChenPerep2, ChenSchrecker}, for a constant $M>0$ depending only on $\ga\in(1,\infty)$, we have the estimates:
\beqa\label{ineq:tildeentropybds}
&\tilde{q}(\rho,m)\geq \frac{1}{M}\big(\rho|u|^3+\rho^{\ga+\th}\big)-M\big(\rho|u|^2+\rho+\rho^\ga\big),\\
&-\check{q}+m(\check{\eta}_\rho+u\check{\eta}_m)\leq 0,\\
&|\check{\eta}_m|\leq M(|u|+\rho^\th),\quad |\check{\eta}_\rho|\leq M(|u|^2+\rho^{2\th}),\\
&|\tilde{\eta}|+\rho|\tilde{\eta}_\rho+u\tilde{\eta}_m|\leq M(\rho |u|^2+\rho+\rho^\ga), 
\eeqa
and, considering $\check{\eta}_\rho+u\check{\eta}_m$ and $\tilde\eta_m$ as functions of $\rho$ and $u$,
\begin{align}
&|(\check{\eta}_\rho+u\check{\eta}_m)_\rho|\leq M(\rho^{\th-1}|u|+\rho^{2\th-1}),\quad &&|(\check{\eta}_\rho+u\check{\eta}_m)_u|\leq M(|u|+\rho^{\th}),\label{ineq:eta_rho+ueta_m}\\
&|(\tilde\eta_m)_\rho|\leq M\rho^{\th-1},\quad &&|(\tilde\eta_m)_u|\leq M.\label{ineq:tildeetamderivatives}
\end{align}
Also, as $\tilde\eta_m=\check\eta_m-\check\eta_m(\bar\rho,0)$, we use \eqref{ineq:tildeentropybds} to verify by Cauchy--Schwarz
\beqa\label{ineq:metam}
|m\tilde\eta_m|\leq&\,M\rho|u|(|u|+\rho^\th+\bar\rho^\th)
\leq  M(\rho u^2+\rho^\ga+\rho\bar\rho^{2\th}).
\eeqa
Moreover, we recall from \cite[Lemma 3.4]{ChenPerep2} that there exists a constant $M>0$, depending only on $\ga>1$, such that for any $(\rho,m)\in\R^2_+$ and $\xi\in\R^2$,
\beq\label{ineq:entropyhessian}|\xi\nabla^2\tilde\eta(\rho,m)\xi^\top|\leq M\xi\nabla^2\eta^*(\rho,m)\xi^\top,\eeq
where $\eta^*(\rho,m)$ is the physical entropy given by
$$\eta^*(\rho,m)=\half\frac{m^2}{\rho}+\frac{\ka\rho^\ga}{\ga-1}.$$
A simple calculation then shows that for smooth functions $\rho$, $m$ with $m=\rho u$,
\beq\label{eq:physicalhessian}
(\rho_r,m_r)\nabla^2\eta^*(\rho,m)(\rho_r,m_r)^\top=\ka\ga\rho^{\ga-2}|\rho_r|^2+\rho|u_r|^2.
\eeq
We also require estimates on the growth of certain $L^p$ norms of the density close to the origin when weighted appropriately.
\begin{lemma}\label{lemma:rho^gammar^l}
There exists $M(\ga)>0$, independent of $\eps$, such that for $l\in\{0,\ldots,n-1\}$, $T>0$,
$$\sup_{t\in[0,T]}\int_r^b\rho(t,y)^\ga y^l\,dy\leq M\big(r^{l+1-n}E_0+\bar\rho^\ga b^{l+1}\big) \text{ for any }r\in[a,b).$$
\end{lemma}
As the proof is similar to that of \cite[Lemma 3.1]{ChenPerep2}, we omit it here. Finally, we recall the following lemma from \cite{ChenPerep2}.
\begin{lemma}[{\cite[Lemma 3.2]{ChenPerep2}}]\label{lemma:rho^3}
There exists a constant $M=M(T)>0$, independent of $\eps$, such that, for any $r\in [a,b)$,
$$\int_0^T\int_r^b \rho(t,y)^3 y^{n-1}\,dy\,dt\leq M\Big(1+\frac{b^n}{\eps}\Big).$$
\end{lemma}
Before we give the proof of the key Lemma \ref{lemma:highinetgrabilityorigin}, we make a couple of remarks about the proof. The key idea is to use the entropy equation for the entropy--entropy flux pair $(\tilde\eta,\tilde{q})$ to gain an estimate on the space--time integral of $\tilde{q}$. Using the lower bound of \eqref{ineq:tildeentropybds} on $\tilde{q}$, we are then able to gain an estimate of the crucial quantity $\rho|u|^3$. In making this estimate, error terms of several types occur. The first arise from the entropy $\tilde\eta$ and are easily controlled up to the origin using the main energy estimate. The second type of error occurs when there is a loss in the radial weight, giving an integrating weight of $r^l$ with $l<n-1$. To handle the apparent loss in such terms, we observe that all such terms may be controlled either by a power of the viscosity or by a factor of $\bar\rho$, which may be taken to zero sufficiently rapidly to provide control. This is the content of assumption \eqref{ass:assumptionintro}. Finally, we must handle the boundary terms appearing at the inner end-point $a(\eps)$ from integration by parts in the viscous terms. The most singular of these occurs as a term growing as $\eps\rho^\ga$. This is handled by a suitable application of the fundamental theorem of calculus and the main energy estimate. Here, we find that the $\eps$ weight is exactly sufficient to provide the required control.
\begin{proof}[Proof of Lemma \ref{lemma:highinetgrabilityorigin}]
We multiply the first equation in \eqref{eq:approx} by $\tilde{\eta}_\rho r^{n-1}$ and the second equation by $\tilde{\eta}_m r^{n-1}$ and sum to obtain
\beqa
\big(r^{n-1}\tilde{\eta}\big)_t+\big(r^{n-1}\tilde{q}\big)_r+(n-1)r^{n-2}\big(-\check{q}+m\check{\eta}_\rho+\frac{m^2}{\rho}\check{\eta}_m+\check{\eta}_m(\bar\rho,0)p(\rho)\big)\\
=\eps r^{n-1}\Big(\big(\rho_{rr}+\frac{n-1}{r}\rho_r\big)\tilde{\eta}_\rho+\big(m_r+\frac{n-1}{r}m\big)_r\tilde{\eta}_m\Big)-r^{n-1}\big(\de\rho^2\big)_r\tilde{\eta}_m .
\eeqa
We integrate this over the region $(0,T)\times(r,b)$ to find
\beqa\label{eq:2.6}
\int_0^T\tilde{q}(t,r)r^{n-1}\,dt&=\int_r^b\big(\tilde{\eta}(T,y)-\tilde{\eta}(0,y)\big)y^{n-1}\,dy+\int_0^T\tilde{q}(\bar\rho,0)b^{n-1}\,dt\\
&+(n-1)\int_0^T\int_r^b\big(-\check{q}+m\check{\eta}_\rho+\frac{m^2}{\rho}\check{\eta}_m\big)y^{n-2}\,dy\,dt\\
&+(n-1)\int_0^T\int_r^b \check{\eta}_m(\bar\rho,0)p(\rho)y^{n-2}\,dy\,dt\\
&-\eps \int_0^T\int_r^b\Big(\big(\rho_{yy}+\frac{n-1}{y}\rho_y\big)\tilde{\eta}_\rho+\big(m_y+\frac{n-1}{y}m\big)_y\tilde{\eta}_m\Big)y^{n-1}dy\,dt\\
&+\de\int_0^T\int_r^b \big(\rho^2\big)_y\tilde{\eta}_m y^{n-1}\,dy\,dt.
\eeqa
Using the upper bound of $|\tilde{\eta}(\rho,m)|\leq \rho |u|^2+\frac{\ka}{\ga-1}\rho^\ga$, the identity $\tilde{q}(\bar\rho,0)=M_0\bar\rho^{\ga+\th}$ for some constant $M_0>0$ and the non-positivity of $-\check{q}+m\check{\eta}_\rho+\frac{m^2}{\rho}\check{\eta}_m$ from \eqref{ineq:tildeentropybds}, we obtain
\beqa\label{eq:aardvark}
\int_0^T\tilde{q}(t,r)r^{n-1}\,dt&\leq ME_0+M_0T\bar\rho^{\ga+\th}b^{n-1}+(n-1)\int_0^T\int_r^b \check{\eta}_m(\bar\rho,0)p(\rho)y^{n-2}\,dy\,dt\\
&-\eps \int_0^T\int_r^b\Big(\big(\rho_{yy}+\frac{n-1}{y}\rho_y\big)\tilde{\eta}_\rho+\big(m_y+\frac{n-1}{y}m\big)_y\tilde{\eta}_m\Big)y^{n-1}\,dy\,dt\\
&-\de\int_0^T\int_r^b \rho^2\big((\tilde{\eta}_m)_\rho\rho_y+(\tilde{\eta}_m)_uu_y\big) y^{n-1}\,dy\,dt\\
&-\de(n-1)\int_0^T\int_r^b \rho^2\tilde{\eta}_m y^{n-2}\,dy\,dt-\de\int_0^T\rho^2\tilde{\eta}_m r^{n-1}\,dt,
\eeqa
by integrating by parts in the final term of \eqref{eq:2.6} and using the boundary condition $\tilde{\eta}_m(\bar\rho,0)=0$.

Now let $\om\in C_c^\infty([0,\infty))$ be as in the statement of the lemma, so that $\om(r)=1$ for $r\in[0,1]$ and $\om(r)\geq0$. We multiply \eqref{eq:aardvark} by $\om(r)$, apply the lower bound $|\tilde{q}(\rho,m)|\geq M^{-1}(\rho|u|^3+\rho^{\ga+\th})-M(\rho|u|^2+\rho+\rho^\ga)$ of \eqref{ineq:tildeentropybds}, and integrate in $r$ from $a$ to $b$ to see
\beqa\label{ineq:ostrich}
\int_0^T\int_a^b\big(\rho |u|^3+\rho^{\ga+\th}\big)\om(r)r^{n-1}\,dr\,dt\leq &\, M(E_0+1+\bar\rho^{\ga+\th}b^{n-1}+I),
\eeqa
where
\beqa
I=&\,(n-1)\int_0^T\int_a^b\int_r^b \check{\eta}_m(\bar\rho,0)p(\rho)y^{n-2}\om(r)\,dy\,dr\,dt\\
&-\eps \int_0^T\int_a^b\int_r^b\Big(\big(\rho_{yy}+\frac{n-1}{y}\rho_y\big)\tilde{\eta}_\rho+\big(m_y+\frac{n-1}{y}m\big)_y\tilde{\eta}_m\Big)y^{n-1}\om(r)\,dy\,dr\,dt\\
&-\de\int_0^T\int_a^b\int_r^b \rho^2\big((\tilde{\eta}_m)_\rho\rho_y+(\tilde{\eta}_m)_uu_y\big) y^{n-1}\om(r)\,dy\,dr\,dt\\
&-\de(n-1)\int_0^T\int_a^b\int_r^b \rho^2\tilde{\eta}_m y^{n-2}\om(r)\,dy\,dr\,dt
-\de\int_0^T\int_a^b\rho^2\tilde{\eta}_m r^{n-1}\om(r)\,dr\,dt\\
=&\,I_1+\cdots+I_5,
\eeqa
and where we have controlled the error term arising from the lower bound on $\tilde{q}$ by
$$\int_0^T\int_a^b(\rho|u|^2+\rho+\rho^\ga)\om(r)\,dr\,dt\leq \int_0^T\int_a^b(\rho|u|^2+\overline{h_\de}(\rho,\bar\rho)+1)\om(r)\,dr\,dt\leq M(E_0+1), $$
using the bound $\rho+\rho^\ga\leq  M(\overline{h_\de}(\rho,\bar\rho)+1)$ of \eqref{ineq:star} and the compact support of $\om(r)$.

We treat $I_1$ first, recalling Lemma \ref{lemma:rho^gammar^l} and \eqref{ineq:tildeentropybds} to bound
\beqas
|I_1|
\leq&\,M\int_0^T\int_a^b\int_r^b\bar\rho^\th\rho^\ga y^{n-2}\om(r)\,dy\,dr\,dt\\
\leq&\, MT\bar\rho^\th\int_a^b \big(r^{-1}+\bar\rho^\ga b^{n-1}\big)\om(r)\,dr\\
\leq&\, MT\bar\rho^\th\big(|\log(a)|+\bar\rho^\ga b^{n-1}\big).
\eeqas
We consider next $I_2$, using integration by parts to re-write the inner integral as:
\beqa\label{ineq:I_21}
I_2=&\,\eps\int_0^T\int_a^b\int_r^b\big(\rho_y(\tilde\eta_\rho)_y+m_y(\tilde\eta_m)_y\big)y^{n-1}\om(r)\,dy\,dr\,dt\\
+&\,\eps(n-1)\int_0^T\int_a^b\int_r^bm\tilde\eta_m y^{n-3}\om(r)\,dy\,dr\,dt+\eps\int_0^T\int_a^b\tilde\eta_r(t,r)r^{n-1}\om(r)\,dr\,dt\\
=&\,I_2^1+I_2^2+I_2^3.
\eeqa
The first term may be expanded as
\beqas
|I_2^1|=&\,\Big|\eps\int_0^T\int_a^b\int_r^b\big(\rho_y^2\tilde\eta_{\rho\rho}+2\rho_ym_y\tilde\eta_{\rho m}+m_y^2\tilde\eta_{mm}\big)y^{n-1}\om(r)\,dy\,dr\,dt\Big|\\
\leq&\,\eps\int_0^T\int_a^b\int_r^b|(\rho_y,m_y)\nabla^2\tilde\eta(\rho,m)(\rho_y,m_y)^\top|y^{n-1}\om(r)\,dy\,dr\,dt\\
\leq&\,M\eps\int_0^T\int_a^b\int_r^b|(\rho_y,m_y)\nabla^2\eta^*(\rho,m)(\rho_y,m_y)^\top|y^{n-1}\om(r)\,dy\,dr\,dt\\
\leq&\,M\eps\int_0^T\int_a^b\int_r^b(\rho^{\ga-2}|\rho_y|^2+\rho|u_y|^2)y^{n-1}\om(r)\,dy\,dr\,dt\\
\leq&\, ME_0,
\eeqas
where we have used the Hessian bound \eqref{ineq:entropyhessian}, the identity \eqref{eq:physicalhessian} and the main energy estimate.

For the second term, we use \eqref{ineq:metam} and \eqref{ineq:star} to bound
$$|m\tilde{\eta}_m|\leq M\big(\rho u^2+\rho^\ga+\rho\bar\rho^{2\th}\big)\leq M\big(\rho|u|^2+\overline{h_\de}(\rho,\bar\rho)+1+\rho\bar\rho^{\ga-1}\big).$$
Thus we find, noting $\rho\bar\rho^{\ga-1}\leq \rho^\ga+\bar\rho^\ga$ by Young's inequality,
\beqa\label{ineq:I_22}
|I_2^2|\leq&\, M\eps \int_0^T\int_a^b\int_r^b \big(\rho|u|^2+\overline{h_\de}(\rho,\bar\rho)+1+\rho\bar\rho^{\ga-1}\big)y^{n-3}\om(r)\,dy\,dr\,dt\\
\leq&\, M\eps\int_0^T\int_a^b\int_r^b \big(\frac{1}{r^2}(\rho|u|^2+\overline{h_\de}(\rho,\bar\rho)+1)y^{n-1}+\frac{1}{r}\bar\rho^{\ga}y^{n-2}\big)\om(r)dy\,dr\,dt\\
\leq&\, M\eps\big(\frac{b^n}{a}+\bar\rho^{\ga}b^{n-1}|\log(a)|\big).
\eeqa
Next, we treat the final term, $I_2^3$, by integrating by parts and using $\om(a)=1$ to find
\beqas
I_2^3=&\,-\eps\int_0^T\int_a^b\tilde\eta(t,r)\big(\om_r(r)+\frac{n-1}{r}\om(r)\big) r^{n-1}\,dr\,dt-\eps\int_0^T\tilde{\eta}(t,a)a^{n-1}\,dt.
\eeqas
Using \eqref{ineq:tildeentropybds}, we easily bound the first term by
$$\Big|\eps\int_0^T\int_a^b\tilde\eta(t,r)\big(\om_r(r)+\frac{n-1}{r}\om(r)\big) r^{n-1}\,dr\,dt\Big|\leq M\frac{\eps}{a}.$$
For the second term, we again apply \eqref{ineq:tildeentropybds} and the boundary condition $u(t,a)=0$ to note that $|\tilde{\eta}(t,a)|\leq M(\rho^\ga+1)$. The contribution from the constant is clearly bounded, so we focus on the $\rho^\ga(t,a)$ term. From the fundamental theorem of calculus and Lemma \ref{lemma:rho^gammar^l}, we obtain
\beqa\label{ineq:I_23}
\eps\int_0^T\rho^\ga(t,a)a^{n-1}dt=&\,-\eps\int_0^T\int_a^b\big(\rho^\ga r^{n-1}\big)_r\,dr\,dt+\eps\int_0^T\bar\rho^\ga b^{n-1}\,dt\\
=&\,-\eps\int_0^T\int_a^b \big(\ga\rho^{\ga-1}\rho_rr^{n-1}+(n-1)\rho^\ga r^{n-2}\big)dr\,dt+\eps T\bar\rho^\ga b^{n-1}\\
\leq&\,M\eps\int_0^T\int_a^b \big(\rho^{\ga-2}|\rho_r|^2+\rho^\ga \big)r^{n-1}\,dr\,dt+M\frac{\eps}{a}+M\bar\rho^\ga b^{n-1}\\
\leq&\, M\big(1+\frac{\eps}{a}+\bar\rho^\ga b^{n-1}\big),
\eeqa
where we have applied the Cauchy-Young inequality and main energy estimate, Proposition \ref{prop:mainenergy}. Thus, combining \eqref{ineq:I_21}--\eqref{ineq:I_23} we have the bound
$$|I_2|\leq M\big(\frac{\eps}{a}+\bar\rho^\ga b^{n-1}+1\big).$$
To bound $I_3$, we recall \eqref{ineq:tildeetamderivatives} and Lemma \ref{lemma:rho^3} to see
\beqas
|I_3|\leq&\, M\de\int_0^T\int_a^b\int_r^b \rho^2\big(\rho^{\th-1}|\rho_y|+|u_y|\big) y^{n-1}\om(r)\,dy\,dr\,dt\\
\leq&\, M\de \int_0^T\int_a^b\int_r^b\big(\rho^{\ga-2}|\rho_y|^2+\rho |u_y|^2+\rho^3\big)y^{n-1}\om(r)\,dy\,dr\,dt\\
\leq&\, M\big(1+\de\frac{b^n}{\eps}\big),
\eeqas
where we have used the main energy estimate and $\de\leq\eps$ to control the derivative terms.

To bound $I_4$, we apply the bound $|\tilde\eta_m|\leq M(|u|+\rho^\th)$ and the Cauchy-Young inequality to show
\beqas
|I_4|\leq&\, M\de\int_0^T\int_a^b\int_r^b\rho^2|\tilde\eta_m|\, y^{n-2}\om(r)\,dy\,dr\,dt\\
\leq&\, M\de\int_0^T\int_a^b\int_r^b\big(\rho^2 |u|+\rho^{2+\th}\big) y^{n-2}\om(r)\,dy\,dr\,dt\\
\leq&\, M\de\int_0^T\int_a^b\int_r^b\big(\rho |u|^2+\rho^3+\rho^{2+\th}\big) y^{n-2}\om(r)\,dy\,dr\,dt.
\eeqas
In the case that $\th\leq1$ (i.e.~$\ga\leq 3$), we then estimate further using Lemma \ref{lemma:rho^3},
\beqas
|I_4|\leq&\,M\de\int_0^T\int_a^b\frac{1}{r}\int_r^b\big(\rho|u|^2+\rho^3+1\big) y^{n-1}\om(r)\,dy\,dr\,dt\\
\leq&\,  M\de|\log(a)|\big(1+\frac{b^n}{\eps}\big).
\eeqas
On the other hand, if $\th>1$ then $\ga>3$ and $\ga>2+\th$, so we use the Cauchy-Young inequality to bound
\beqas
|I_4|\leq&\,M\de\int_0^T\int_a^b\frac{1}{r}\int_r^b\big(\rho|u|^2+\rho^{\ga}+1\big) y^{n-1}\om(r)\,dy\,dr\,dt\\
\leq&\, M\de|\log(a)|\big(1+(\bar\rho^\ga+1)b^n\big).\eeqas
Finally, $I_5$ is treated analogously to $I_4$, giving a bound of
$$|I_5|\leq \de M\big(1+\frac{b^n}{\eps}\big).$$ 
By \eqref{ass:assumptionintro}, all of the above bounds for the terms $I_1,\ldots,I_5$ become uniform with respect to $\eps$, hence we conclude from \eqref{ineq:ostrich} (and the obvious estimate $(\bar\rho^\ga+1)\leq\eps^{-1}$) that
\beqas
\int_0^T\int_a^b&\big(\rho |u|^3+\rho^{\ga+\th}\big)\om(r)r^{n-1}\,dr\,dt\\
&\leq M\Big(\bar\rho^\th|\log(a)|+\eps\frac{b^n}{a}+\bar\rho^\ga b^{n-1}+\frac{\eps}{a}+\de|\log(a)|\big(1+\frac{b^n}{\eps}\big)+1\Big)\leq M,
\eeqas
and so we conclude the proof of the lemma.
\end{proof}

\section{Viscous terms}\label{sec:viscosity}
We begin this section with the two main estimates we need to demonstrate convergence to zero of the viscous terms in the weak formulation of the approximate equations, system \eqref{eq:approx}.
\begin{lemma}\label{lemma:rho_r}
Let $\om=\om(r)\in C^\infty_c(\R)$ be a test function such that $\om(r)=1$ for $r\in[0,1]$, $\om(r)=0$ for $r\geq 2$. Then for any $\De\in(0,\half)$, there exists a constant $M>0$, independent of $\De$ and $\eps$, such that
\beq
\eps\int_0^T\int_a^b \rho_r^2\mathbbm{1}_{\{\rho<\De\}}\om(r)^2r^{n-1}\,dr\,dt\leq M\big(\sqrt{\eps}(1+\De^{4-\ga})+\frac{\De}{a}+\frac{\De^{3/2}}{\sqrt{\varepsilon}}\big).
\eeq
\end{lemma}

\begin{lemma}\label{lemma:rho_r/rho}
Let $\om=\om(r)\in C^\infty_c(\R)$ be a test function such that $\om(r)=1$ for $r\in[0,1]$, $\om(r)=0$ for $r\geq 2$. Then for any $\De\in(0,\half)$, there exists a constant $M>0$, independent of $\De$ and $\eps$, such that
\beq
\eps\int_0^T\int_a^b\frac{\rho_r^2}{\rho}\mathbbm{1}_{\{\rho<\De\}}\om(r)^2 r^{n-1}dr\,dt\leq M\big(|\log\De|+\frac{\sqrt{\De}}{\sqrt{\eps}}+\frac{\sqrt{\De}}{a}+\sqrt{\eps}|\log\De|\De^{\frac{2-\ga}{2}}\big).
\eeq
\end{lemma}

The proofs of these two lemmas are motivated by the following observation. Let $\varphi=\varphi(\rho)$ be a twice differentiable function, $\om=\om(r)\in C_c^\infty(\R)$, and multiply the first equation in \eqref{eq:approx} by $\varphi'(\rho)\om(r)^2$. A simple calculation yields
\beqa
\big(r^{n-1}\varphi\om^2\big)_t&+\big(r^{n-1}\varphi u \om^2\big)_r+r^{n-1}\big(\rho\varphi'-\varphi\big)\big(u_r+\frac{n-1}{r}u\big)\om^2-2r^{n-1}\varphi u\om\om_r\\
&=\eps\big(r^{n-1}\varphi'\rho_r\om^2\big)_r-\eps r^{n-1}\varphi''\rho_r^2\om^2-2\eps r^{n-1}\varphi_r\om\om_r.
\eeqa
Thus, for any such $\varphi$, 
\beqa\label{eq:continuityequationtestfunction}
\eps\int_0^T&\int_a^b\varphi''\rho_r^2\om^2r^{n-1}\,dr\,dt\\
=&\,-\int_a^b\varphi\om^2r^{n-1}\Big|_0^T\,dr+\int_0^T\int_a^b\big(\varphi-\rho\varphi'\big)u_r\om^2r^{n-1}\,dr\,dt\\
&+2\int_0^T\int_a^b\varphi u\om\om_rr^{n-1}\,dr\,dt+\int_0^T\int_a^b\frac{n-1}{r}\big(\varphi-\rho\varphi'\big)u\om^2r^{n-1}\,dr\,dt\\
&-2\eps\int_0^T\int_a^b \varphi_r\om\om_rr^{n-1}\,dr\,dt,
\eeqa
where we have used the boundary conditions $\rho_r=u=0$ at $a$ and the compact support of $\om$.

\begin{proof}[Proof of Lemma \ref{lemma:rho_r}]
We define, for $\De\in(0,\half)$ fixed,
\beq
\varphi(\rho)=\begin{cases}
\frac{\rho^2}{2}, &\rho<\De,\\
\frac{\De^2}{2}+\De(\rho-\De), &\rho\geq \De.
\end{cases}
\eeq
Then we have that
\beqas
&\varphi''(\rho)=\mathbbm{1}_{\{\rho<\De\}}(\rho),\\
&\varphi(\rho)-\rho\varphi'(\rho)=-\half\min\{\rho^2,\De^2\}.
\eeqas
Then from \eqref{eq:continuityequationtestfunction}, we obtain
\beqa\label{eq:antelope}
\eps\int_0^T&\int_a^b\rho_r^2\mathbbm{1}_{\{\rho<\De\}}\om^2r^{n-1}\,dr\,dt\\
=&\,-\int_a^b\varphi\om^2r^{n-1}\Big|_0^T\,dr-\half\int_0^T\int_a^b\min\{\rho^2,\De^2\}u_r\om^2r^{n-1}\,dr\,dt\\
&+2\int_0^T\int_a^b\varphi u\om\om_rr^{n-1}\,dr\,dt-\half\int_0^T\int_a^b\frac{n-1}{r}\min\{\rho^2,\De^2\}u\om^2r^{n-1}\,dr\,dt\\
&-2\eps\int_0^T\int_a^b \varphi_r\om\om_rr^{n-1}\,dr\,dt\\
=&\,J_1+\cdots+J_5.
\eeqa
To bound $J_1$, we simply observe that $|\varphi(\rho)|\leq \De\rho$ for all $\rho>0$. Thus
\beqs
|J_1|\leq M\sup_{t\in[0,T]}\int_a^b \De\rho\om^2r^{n-1}\,dr\leq M\De
\eeqs
by the main energy estimate, Proposition \ref{prop:mainenergy}, where we have used the compact support of $\om$ and the estimate $\rho\leq M\big(1+\overline{h_\de}(\rho,\bar\rho)\big)$ of \eqref{ineq:star}.

The next simplest term to control is $J_3$, which we bound in a similar way, giving an estimate of $$|J_3|\leq M\int_0^T\int_a^b\De(\rho+\rho u^2)\om r^{n-1}\,dr\,dt\leq M \De,$$
where we again use the main energy estimate and $M$ depends on  $|\supp\,\om|$ and $\|\om_r\|_{L^\infty}$.

Turning now to $J_2$, we estimate
\beqas
|J_2|\leq&\, \frac{1}{2\sqrt{\eps}}\int_0^T\int_a^b \De^{3/2}\sqrt{\eps}\sqrt{\rho}|u_r|\om^2r^{n-1}\,dr\,dt\\
\leq&\, \frac{M\De^{3/2}}{\sqrt{\eps}}\Big(\eps\int_0^T\int_a^b \rho u_r^2r^{n-1}\,dr\,dt\Big)^\half\\
\leq&\, \frac{M\De^{3/2}}{\sqrt{\eps}},
\eeqas
by the main energy estimate, where $M$ also depends on $|\supp\,\om|$.

Next, we use that $r>a$ in the domain of integration and Proposition \ref{prop:mainenergy} to bound
\beqs
|J_4|\leq M\int_0^T\int_a^b\frac{n-1}{r}\De\rho u\om^2 r^{n-1}dr\,dt\leq M\frac{\De}{a}\int_0^T\int_a^b (\rho +\rho u^2)\om^2r^{n-1}dr\,dt\leq \frac{M\De}{a}.
\eeqs
We consider $J_5$ on the two regions $\{\rho<\De\}$ and $\{\rho\geq\De\}$ by writing
\beqas
J_5=&-\,2\eps\int_0^T\int_a^b \min\{\rho,\De\}\rho_r\om\om_rr^{n-1}\,dr\,dt\\
=&-\,2\eps \int_0^T\int_a^b \rho\mathbbm{1}_{\{\rho<\De\}}\rho_r\om\om_rr^{n-1}\,dr\,dt -2\eps\int_0^T\int_a^b \De\mathbbm{1}_{\{\rho\geq\De\}}\rho_r\om\om_rr^{n-1}\,dr\,dt\\
=&\,J_5^1+J_5^2.
\eeqas
Considering the second term first, we use the Cauchy-Young inequality to bound
\beqas
|J_5^2|\leq&\,2\sqrt{\varepsilon}\int_0^T\int_a^b\sqrt{\varepsilon}\De\rho^{\frac{\gamma-2}{2}}\rho^{\frac{2-\gamma}{2}}|\rho_r|\mathbbm{1}_{\{\rho\geq\Delta\}} \om\om_rr^{n-1}\,dr\,dt\\
\leq&\,M\sqrt{\varepsilon}\int_0^T\int_a^b\varepsilon\rho^{\gamma-2}|\rho_r|^2r^{n-1}dr\,dt+M\Delta^2\sqrt{\varepsilon}\int_0^T\int_a^b \rho^{2-\gamma}\mathbbm{1}_{\{\rho\geq\Delta\}}\om r^{n-1}dr\,dt.
\eeqas
In the case that $\ga\in(1,2]$, we make the estimate $\rho^{2-\ga}\leq\rho^\ga+1\leq M(\overline{h_\de}(\rho,\bar\rho)+1)$ by \eqref{ineq:star} and apply the main energy estimate to obtain
$$|J_5^2|\leq M\sqrt{\eps}.$$
On the other hand, for $\ga>2$, we estimate $\rho^{2-\ga}\leq \De^{2-\ga}$ on the region $\rho\geq\De$  to obtain 
$$|J_5^2|\leq M\sqrt{\eps}\big(1+\De^{4-\ga}\big),$$
where we have used the main energy estimate to bound the first term of $J_5^2$.

Turning finally to $J_5^1$, we use the Cauchy-Young inequality to estimate
\beqas
|J_5^1|\leq&\,\frac{\varepsilon}{2}\int_0^T\int_a^b|\rho_r|^2\mathbbm{1}_{\{\rho<\Delta\}}\omega^2r^{n-1}\,dr\,dt+M\varepsilon\int_0^T\int_a^b\Delta^2\om_r^2r^{n-1}\,dr\,dt\\
\leq&\,\frac{\varepsilon}{2}\int_0^T\int_a^b|\rho_r|^2\mathbbm{1}_{\{\rho<\Delta\}}\omega^2r^{n-1}\,dr\,dt+M\varepsilon\De^2.
\eeqas
Combining this with the estimate above for $J_5^2$, we obtain
$$|J_5|\leq \frac{\varepsilon}{2}\int_0^T\int_a^b|\rho_r|^2\mathbbm{1}_{\{\rho<\Delta\}}\omega^2r^{n-1}\,dr\,dt+M\varepsilon\De^2+M\sqrt{\eps}\big(1+\De^{4-\ga}\big).$$
 Thus, combining the estimates for $J_1,\ldots,J_5$  in \eqref{eq:antelope},
$$
\hspace{17.5mm}\eps\int_0^T\int_a^b\rho_r^2\mathbbm{1}_{\{\rho<\De\}}\om^2r^{n-1}\,dr\,dt\leq M\big(\sqrt{\eps}(1+\De^{4-\ga})+\frac{\De}{a}+\frac{\De^{3/2}}{\sqrt{\varepsilon}}\big).\hspace{17.5mm}\qedhere
$$
 \end{proof}

\begin{proof}[Proof of Lemma \ref{lemma:rho_r/rho}]
We let $\De\in(0,\half)$ and define the function $\psi(\rho)$ by
\beq
\psi(\rho)=\begin{cases}
\rho\log\rho-\rho, &\rho<\De,\\
\rho\log\De-\De, &\rho\geq\De,
\end{cases}
\eeq
so that \beqa
&\psi(\rho)-\rho\psi'(\rho)=-\min\{\rho,\De\},\\
&\psi''(\rho)=\frac{1}{\rho}\mathbbm{1}_{\{\rho<\De\}}.
\eeqa
Then from \eqref{eq:continuityequationtestfunction}, we obtain
\beqa\label{eq:kangaroo}
\eps\int_0^T\int_a^b&\frac{\rho_r^2}{\rho}\mathbbm{1}_{\{\rho<\De\}}\om^2r^{n-1}\,dr\,dt\\
=&\,-\int_a^b\psi\om^2r^{n-1}\Big|_0^T\,dr-\int_0^T\int_a^b\min\{\rho,\De\}u_r\om^2r^{n-1}\,dr\,dt\\
&+2\int_0^T\int_a^b\psi u\om\om_rr^{n-1}\,dr\,dt-\int_0^T\int_a^b\frac{n-1}{r}\min\{\rho,\De\}u\om^2r^{n-1}\,dr\,dt\\
&-2\eps\int_0^T\int_a^b \psi_r\om\om_rr^{n-1}\,dr\,dt\\
=&\,J_1+\cdots+J_5.
\eeqa
As $\rho|\log\rho|\mathbbm{1}_{\{\rho<\De\}}\leq\De|\log\De|$ for $\De\in(0,\half)$, we bound $J_1$ by
\beqs
|J_1|\leq M\sup_{t\in[0,T]}\int_a^b\big(\De|\log\De|+\rho|\log\De|\big)\om^2r^{n-1}\,dr\leq  M|\log\De|,
\eeqs
where we have employed the main energy estimate, Proposition \ref{prop:mainenergy}, and $\rho\leq M(1+\overline{h_\de}(\rho,\bar\rho))$, again by \eqref{ineq:star}.

$J_3$ is bounded similarly, using the estimate $\rho u\log\rho\leq \rho u^2+\rho(\log\rho)^2\leq \rho u^2+\De(\log\De)^2$ for $\rho<\De$, giving an estimate of $$|J_3|\leq M|\log\De|,$$ where $M$ depends on $\|\om_r\|_{L^\infty}$.

To control $J_2$, we again employ the main energy estimate and H\"older's inequality to obtain
\beqas
|J_2|\leq&\, \int_0^T\int_a^b \min\{\rho,\De\}|u_r|\om^2r^{n-1}\,dr\,dt\\
\leq&\, M\sqrt{\De}\Big(\int_0^T\int_a^b \rho|u_r|^2r^{n-1}\,dr\,dt\Big)^\half
\leq M\frac{\sqrt{\De}}{\sqrt{\eps}}.
\eeqas
For $J_4$, we use the Cauchy-Schwarz inequality to bound
\beqas
|J_4|\leq&\, \int_0^T\int_a^b\frac{n-1}{r}\min\{\rho,\De\}|u|\om^2r^{n-1}\,dr\,dt\\
\leq&\,M\frac{\sqrt{\De}}{a}\Big(\int_0^T\int_a^b \rho |u|^2r^{n-1}\,dr\,dt\Big)^\half\\
\leq&\, M\frac{\sqrt{\De}}{a}.
\eeqas
Finally, we break $J_5$ into two terms, one supported on the region $\{\rho<\De\}$ and the other on the region $\{\rho\geq\De\}$,
\beqas J_5=&-2\eps\int_0^T\int_a^b\big(\log\rho\mathbbm{1}_{\{\rho<\De\}}+\log\De\mathbbm{1}_{\{\rho\geq\De\}}\big)\rho_r\om\om_rr^{n-1}\,dr\,dt\\
=&\,J_5^1+J_5^2.
\eeqas
Estimating the first term, we use the Cauchy-Young inequality to bound
\beqas
|J_5^1|=&\,2\eps\,\Big|\int_0^T\int_a^b \log\rho\mathbbm{1}_{\{\rho<\De\}}\rho_r\om\om_rr^{n-1}\,dr\,dt\Big|\\
\leq&\, \frac{\eps}{2}\int_0^T\int_a^b \frac{\rho_r^2}{\rho}\mathbbm{1}_{\{\rho<\De\}}\om^2r^{n-1}\,dr\,dt+M\eps\int_0^T\int_a^b\rho(\log\rho)^2\mathbbm{1}_{\{\rho<\De\}}\om_r^2r^{n-1}\,dr\,dt\\
\leq&\,\frac{\eps}{2}\int_0^T\int_a^b \frac{\rho_r^2}{\rho}\mathbbm{1}_{\{\rho<\De\}}\om^2r^{n-1}\,dr\,dt+M\eps\\
\eeqas
where we have also applied the main energy estimate, Proposition \ref{prop:mainenergy}.

For the last term, we make the bound
\beqas
|J_5^2|=&\,2\eps\,\Big|\int_0^T\int_a^b\log\De\mathbbm{1}_{\{\rho\geq\De\}}\rho_r\om\om_rr^{n-1}\,dr\,dt\Big|\\
\leq&\,\eps|\log\De|\int_0^T\int_a^b |\rho_r|\mathbbm{1}_{\{\rho\geq\De\}}\om|\om_r|r^{n-1}\,dr\,dt\\
\leq&\,M|\log\De|\sqrt{\eps}\Big(\eps\int_0^T\int_a^b\max\{\De^{2-\ga},1\}\rho^{\ga-2}|\rho_r|^2\om(r)^2r^{n-1}\,dr\,dt\Big)^\half,\\
\leq&\,M\sqrt{\eps}|\log\De|(1+\De^{\frac{2-\ga}{2}})
\eeqas 
where we have again used the energy estimate of Proposition \ref{prop:mainenergy}.

Combining the estimates for $J_1,\ldots,J_5$ in \eqref{eq:kangaroo}, we conclude the proof.
\end{proof}

\section{Proof of Theorem \ref{theorem}}\label{sec:mainproof}
We begin by recalling the following theorem from \cite{ChenPerep2, ChenSchrecker}.
\begin{theorem}\label{theorem:convergence}
Let $(\rho_0,m_0)\in L^1_{loc}(\R_+)^2$, $\rho_0\geq0$, be of finite energy,
$$E[\rho_0,m_0]<\infty,$$
and suppose that for $\eps>0$, the parameters $\bar\rho(\eps)$, $\de(\eps)$, $b(\eps)$ satisfy
\beqs
\bar\rho^\ga b^n+\frac{\de}{\eps}b^n\leq M,
\eeqs
where $M$ is independent of $\eps$. Let $(\rho_0^\eps,m_0^\eps)$ with $\inf\rho_0^\eps=c_\eps>0$ (where $c_\eps$ does not need to be uniform in $\eps$) be smooth functions on $(a(\eps),b(\eps))$ such that
\begin{itemize}
\item $(\rho_0^\eps,m_0^\eps)\to(\rho_0,m_0)$ for almost every $r\in\R_+$ as $\eps\to0$, where we extend $(\rho_0^\eps,m_0^\eps)$ from $(a,b)$ to $\R_+$ by zero;
\item $(\rho_0^\eps,m_0^\eps)$ satisfies the boundary conditions \eqref{eq:approxbcs} as well as the  compatibility conditions:
\beqas
&\big(r^{n-1}m_0^\eps\big)_r\Big|_{r=a}=0,\\
&\Big(r^{n-1}m_{0,r}^\eps-\eps\big(r^{n-1}\rho_{0,r}^\eps\big)_r\Big)\Big|_{r=b}=\Big(r^{n-1}\Big(\frac{(m_0^\eps)^2}{\rho_0^\eps}+p_\de(\rho_0^\eps)\Big)_r-\eps \big(r^{n-1}m_0^\eps\big)_r\Big)\Big|_{r=b}=0;
\eeqas
\item $E[\rho_0^\eps,m_0^\eps]\to E[\rho_0,m_0]$ as $\eps\to0$.
\end{itemize}
Then there exist unique classical solutions $(\rho^\eps,m^\eps)$ of \eqref{eq:approx}--\eqref{eq:approxbcs} {\rm (}extended by $0$ to $\R^2_+${\rm )} which converge $(\rho^\eps,m^\eps)\to(\rho,m)$ almost everywhere in $\R_+^2$ and in $L^p_{loc}(\R^2_+)\times L^q_{loc}(\R^2_+)$ for $p\in[1,\ga+1)$ and $q\in[1,\frac{3(\ga+1)}{\ga+3})$.
\end{theorem}
We strengthen the assumptions of Theorem \ref{theorem:convergence} by imposing assumption \eqref{ass:assumptionintro}, as well as the slightly stronger condition, guaranteed by appropriate choice of $a$,
\beq \label{ass:assumption}
\frac{\sqrt{\eps}}{a(\eps)}\to0 \text{ as } \eps\to0.
\eeq
Now we let $\varphi\in C_c^\infty([0,\infty)^2)$, multiply the first equation in \eqref{eq:approx} by $\varphi$ and integrate by parts on $[0,T]\times(a(\eps),b(\eps))$, using the boundary conditions \eqref{eq:approxbcs}, to obtain
\beqa\label{eq:weakformapproxcontinuity}
\int_0^T&\int_a^b \big(\rho^\eps(t,r)\varphi_t(t,r)+m^\eps(t,r)\varphi_r(t,r)\big)r^{n-1}\,dr\,dt+\int_a^b\rho^\eps_0(r)\varphi(0,r) r^{n-1}\,dr\\
=&\,\eps\int_0^T\int_a^b \rho^\eps_r(t,r)\varphi_r(t,r)r^{n-1}\,dr\,dt.
\eeqa
As $\varphi$ has compact support in $[0,\infty)^2$, we may apply the uniform bound of Lemma \ref{lemma:highinetgrabilityorigin} and the almost everywhere convergence  $(\rho^\eps,m^\eps)\to(\rho,m)$ to deduce that the left hand side of \eqref{eq:weakformapproxcontinuity} converges as $\eps\to0$ to 
\beqs
\int_0^T\int_a^b \big(\rho(t,r)\varphi_t(t,r)+m(t,r)\varphi_r(t,r)\big)r^{n-1}\,dr\,dt+\int_a^b\rho_0(r)\varphi(0,r) r^{n-1}\,dr.
\eeqs
To control the dissipative term, we distinguish between the two cases $\ga>2$ and $\ga\leq 2$. When $\ga\leq 2$, we make the estimate $\rho^{2-\ga}\leq\rho^\ga+1\leq M(\overline{h_\de}(\rho^\eps,\bar\rho)+1)$ and apply the Cauchy--Young inequality to see
\beqas
\Big|\eps&\int_0^T\int_a^b \rho^\eps_r(t,r)\varphi_r(t,r)r^{n-1}\,dr\,dt\Big|\\
\leq&\int_0^T\int_a^b\big(\eps^{3/2}(\rho^\eps)^{\ga-2}|\rho^\eps_r|^2+\eps^{1/2}\rho^{2-\ga}\big)|\varphi_r(t,r)|r^{n-1}\,dr\,dt\\
\leq&\, M\eps^{1/2}\Big(\eps\int_0^T\int_a^b(\rho^\eps)^{\ga-2}|\rho^\eps_r|^2r^{n-1}\,dr\,dt+\int_0^T\int_a^b (\overline{h_\de}(\rho^\eps,\bar\rho)+1)|\phi_r|r^{n-1}\,dr\,dt\Big)\\
\leq&\, M\eps^{1/2},
\eeqas
which tends to zero as $\eps\to0$.

On the other hand, when $\ga>2$, we fix some $\De\in(0,\half)$ (which does not change) and use Lemma \ref{lemma:rho_r} to estimate
\beqas
\Big|\eps&\int_0^T\int_a^b \rho^\eps_r(t,r)\varphi_r(t,r)r^{n-1}\,dr\,dt\Big|\\
\leq&\,\eps\int_0^T\int_a^b |\rho^\eps_r(t,r)\varphi_r(t,r)(\mathbbm{1}_{\{\rho<\De\}}+\mathbbm{1}_{\{\rho\geq\De\}})|r^{n-1}\,dr\,dt\\
\leq&\, M\Big(\eps^2\int_0^T\int_a^b |\rho^\eps_r|^2|\varphi_r|\mathbbm{1}_{\{\rho<\De\}}r^{n-1}\,dr\,dt\Big)^\half\\
&+\Big(\eps^2\De^{2-\ga}\int_0^T\int_a^b (\rho^\eps)^{\ga-2}|\rho^\eps_r|^2|\varphi_r|\mathbbm{1}_{\{\rho\geq \De\}}r^{n-1}\,dr\,dt\Big)^\half\\
\leq&\, M_\De(\frac{\eps}{a}+\sqrt{\eps})^{\half}\to0
\eeqas
as $\eps\to0$ by \eqref{ass:assumption}, where $M_\De$ depends on $\varphi_r$ and the fixed constant $\De$, thus demonstrating \eqref{eq:weakformcty}.

Let now $\varphi\in C_c^\infty([0,\infty)^2)$ be such that $\varphi(t,0)=0$ and take a sequence $\{\varphi^\eps\}_{\eps>0}$ in $C_c^\infty(\R^2_+)$, uniformly bounded in $W^{1,\infty}(\R^2_+)$, such that $\varphi^\eps\to\varphi$ strongly in $W^{1,p}(\R^2_+)$ for all $p<\infty$ and $\varphi^\eps(t,r)=0$ for $r\in[0,a(\eps)]$ and $t\in[0,T]$. We choose the sequence $\varphi^\eps$ such that the supports of the $\varphi^\eps$ are contained in a fixed compact set in $[0,\infty)^2$. We multiply the second equation in \eqref{eq:approx} by $\varphi^\eps$ and integrate by parts on $[0,T]\times(a(\eps),b(\eps))$, using the boundary conditions \eqref{eq:approxbcs} and $\varphi^\eps(t,a)=0$, to obtain
\beqa\label{eq:weakformapproxmomentum}
&\int_0^T\hspace{-0.5mm}\int_a^b \Big(m^\eps\varphi^\eps_t+\frac{(m^\eps)^2}{\rho^\eps}\varphi^\eps_r +p_\de(\rho^\eps)\big(\varphi^\eps_r+\frac{n-1}{r}\varphi^\eps\big)\Big)r^{n-1}\,dr\,dt\\&+\int_a^b\hspace{-0.5mm} m_0^\eps(r)\varphi^\eps(0,r)r^{n-1}\,dr
=\eps\int_0^T\int_a^b \Big(\big(r^{n-1}m^\eps\big)_r\varphi^\eps_{r}+\frac{n-1}{r}(r^{n-1}m^\eps)_r\varphi^\eps\Big)dr\,dt.
\eeqa
Then, using again Lemma \ref{lemma:highinetgrabilityorigin} and the uniform compact support of $\varphi^\eps$, we see the convergence
\beqas
\lim_{\eps\to0}&\int_0^T\int_a^b \Big(m^\eps\varphi^\eps_t+\frac{(m^\eps)^2}{\rho^\eps}\varphi^\eps_r +p_\de(\rho^\eps)\big(\varphi^\eps_r+\frac{n-1}{r}\varphi^\eps\big)\Big)r^{n-1}\,dr\,dt\\
&=\int_0^T\int_a^b \Big(m\varphi_t+\frac{m^2}{\rho}\varphi_r +p(\rho)\big(\varphi_r+\frac{n-1}{r}\varphi\big)\Big)r^{n-1}\,dr\,dt,
\eeqas
where for the final term, $p_\de(\rho^\eps)\frac{n-1}{r}\varphi^\eps r^{n-1}$, we note $|\varphi^\eps(t,r)|\leq r\|\varphi^\eps_r\|_{L^\infty}\leq Mr$, so 
$$p_\de(\rho^\eps)\frac{n-1}{r}\varphi^\eps r^{n-1}\leq M p_\de(\rho^\eps)r^{n-1}$$
on the support of $\varphi$. Hence by Lemma \ref{lemma:highinetgrabilityorigin}, this term is also equi-integrable and so converges.
Considering now the right hand side of \eqref{eq:weakformapproxmomentum}, we integrate by parts in the final term to see
\beqas
\eps\int_0^T\int_a^b \Big(\big(r^{n-1}m^\eps\big)_r\varphi^\eps_{r}&+\frac{n-1}{r}(r^{n-1}m^\eps)_r\varphi^\eps\Big)dr\,dt\\
&=\eps\int_0^T\int_a^b \big(m^\eps_r\varphi^\eps_r+\frac{n-1}{r^2}m^\eps\varphi^\eps\big)r^{n-1}dr\,dt.
\eeqas
For the last term, as $\rho^\eps\leq (\rho^\eps)^\ga+1$ and $\supp\,\varphi^\eps$ is compact, we note by Proposition \ref{prop:mainenergy} that
\beqas \Big|\eps\int_0^T\int_a^b\frac{n-1}{r^2}m^\eps\varphi^\eps r^{n-1}dr\,dt\Big|&\leq \eps\int_0^T\int_a^b(\rho^\eps+\rho^\eps|u^\eps|^2)r^{n-3}\varphi^\eps dr\,dt\\
&\leq M\eps a^{-2}+M\eps,\eeqas
which converges to 0 as $\eps\to0$ by \eqref{ass:assumption}.

For the remaining term, we apply H\"older's inequality to bound
\beqas
\Big|\eps\int_0^T\int_a^b &m^\eps_r\varphi^\eps_rr^{n-1}\,dr\,dt\Big|\\
=&\,\Big|\eps\int_0^T\int_a^b (\rho^\eps u^\eps_r+\rho^\eps_ru^\eps)\varphi^\eps_rr^{n-1}\,dr\,dt\Big|\\
\leq& \Big(\eps\int_0^T\int_a^b\rho^\eps |u^\eps_r|^2r^{n-1}\,dr\,dt\Big)^\half\Big(\eps\int_0^T\int_a^b\rho^\eps|\varphi^\eps_r|^2 r^{n-1}\,dr\,dt\Big)^\half\\
+&\Big(\eps^{\frac{3}{2}}\int_0^T\int_a^b\frac{(\rho_r^\eps)^2}{\rho^\eps}|\varphi^\eps_r|^2r^{n-1}\,dr\,dt\Big)^\half\Big(\eps^{\frac{1}{2}}\int_0^T\int_a^b\rho^\eps (u^\eps)^2r^{n-1}\,dr\,dt\Big)^\half
\eeqas
which converges to 0 as $\eps\to0$ by the main energy estimate, Proposition \ref{prop:mainenergy}, and Lemmas \ref{lemma:rho_r} and \ref{lemma:rho_r/rho},
thus demonstrating \eqref{eq:weakformmomentum} and hence concluding the proof of Theorem \ref{theorem}.

\appendix
\section{}
For the sake of clarity and the convenience of the reader, we provide here a derivation of the weak formulation for spherically symmetric gas dynamics and comment on the conditions at the origin, $r=0$. This derivation may also be found in, for example, \cite[Theorem 5.7]{Hoff}. We focus our exposition here on the momentum equations as similar considerations hold for the continuity equation. Recall from Definition \ref{def:multiDsoln} that the weak formulation, in $\R^n$, for the momentum equation is:\\
For each $j=1,\dots,n$, and $\varphi\in C_c^\infty([0,\infty)\times\R^n;\R)$,
\beqs
\int_0^\infty\int_{\R^n}\big(\rho\mathbf{u}^j\mathbf{\varphi}_t+(\rho\mathbf{u}^j\mathbf{u})\cdot \nabla_{\mathbf{x}}\mathbf{\varphi}+p(\rho)\partial_{x_j}\mathbf{\varphi}\big)\,d{\mathbf{x}}\,dt+\int_{\R^n}\rho_0\mathbf{u}^j_0(\mathbf{x})\mathbf{\varphi}(0,{\mathbf{x}})\,d{\mathbf{x}}=0,
\eeqs
where $\mathbf{u}^j$ denotes the $j$-th component of the vector field $\mathbf{u}$.

Thus for a spherically symmetric motion, as
\beqs
\rho(t,{\mathbf{x}})=\rho(t,r),\qquad \mathbf{u}(t,{\mathbf{x}})=u(t,r)\frac{{\mathbf{x}}}{|{\mathbf{x}}|},
\eeqs
we may re-write this weak formulation as follows. For the first term, we see that
\beqas
\int_0^\infty\int_{\R^n}\rho\mathbf{u}^j\mathbf{\varphi}_t\,d{\mathbf{x}}\,dt=&\,\int_0^\infty\int_0^\infty\int_{|\mathbf{x}|=r}\rho(t,r)u(t,r)\frac{x_j}{r}\varphi_t(t,\mathbf{x})\,dS_\mathbf{x}\,dr\,dt\\
=&\,\int_0^\infty\int_0^\infty\int_{|\mathbf{y}|=1}\rho(t,r)u(t,r){y_j}\varphi_t(t,r\mathbf{y})\,dS_{\mathbf{y}}r^{n-1}\,dr\,dt\\
=&\,\int_0^\infty\int_0^\infty\rho(t,r)u(t,r)\zeta_t(t,r)r^{n-1}\,dr\,dt,
\eeqas
where we have defined the new test function
$$\zeta(t,r)=\int_{|\mathbf{y}|=1}y_j\varphi(t,r\mathbf{y})\,dS_\mathbf{y}.$$
Similarly, we find that
$$\int_{\R^n}\rho_0\mathbf{u}^j_0(\mathbf{x})\mathbf{\varphi}(0,{\mathbf{x}})\,d{\mathbf{x}}=\int_0^\infty\rho_0(r)u_0(r)\zeta(0,r)r^{n-1}\,dr.$$
For the next term, we calculate
\beqas
\int_0^\infty\int_{\R^n}\rho\mathbf{u}^j\mathbf{u}\cdot \nabla_{\mathbf{x}}\mathbf{\varphi}\,d\mathbf{x}\,dt=&\,\int_0^\infty\int_0^\infty\int_{|\mathbf{y}|=1}\rho u^2y_j\mathbf{y}\cdot\nabla\varphi(t,r\mathbf{y})\,dS_\mathbf{y}r^{n-1}\,dr\,dt\\
=&\,\int_0^\infty\int_0^\infty\rho u^2\zeta_rr^{n-1}\,dr\,dt.
\eeqas
For the final term, we first make the observation that
\beqas
r^{n-1}\int_{|\mathbf{y}|=1}\varphi_{x_j}(t,r\mathbf{y})\,dS_\mathbf{y}=&\,\frac{\partial}{\partial r}\int_0^r\Big(\int_{|\mathbf{y}|=1}\varphi_{x_j}(t,\tilde r\mathbf{y})\,dS_\mathbf{y}\Big)\tilde r^{n-1}\,d\tilde r\\
=&\,\frac{\partial}{\partial r}\int_{|\mathbf{x}|\leq r}\varphi_{x_j}(t,\mathbf{x})\,d\mathbf{x}\\
=&\,\frac{\partial}{\partial r}\int_{|\mathbf{x}|=r}\varphi(t,\mathbf{x})\frac{x_j}{|\mathbf{x}|}\,dS_\mathbf{x}\\
=&\,\frac{\partial}{\partial r}(r^{n-1}\zeta).
\eeqas
Then we check
\beqas
\int_0^\infty\int_{\R^n}p(\rho)\partial_{x_j}\mathbf{\varphi}\,d{\mathbf{x}}\,dt=&\,\int_0^\infty\int_0^\infty\int_{|\mathbf{y}|=1}p(\rho)\varphi_{x_j}(t,r\mathbf{y})\,dS_\mathbf{y}r^{n-1}\,dr\,dt\\
=&\,\int_0^\infty\int_0^\infty p(\rho)(r^{n-1}\zeta)_r\,dr\,dt\\
=&\,\int_0^\infty\int_0^\infty p(\rho)\big(\zeta_r+\frac{n-1}{r}\zeta\big)r^{n-1}\,dr\,dt.
\eeqas
Putting all of these identities together, we obtain the equivalent weak formulation 
\beqas
\int_0^\infty\int_0^\infty \Big(\rho u\zeta_t+\rho u^2\zeta_r+p(\rho)\big(\zeta_r+&\frac{n-1}{r}\zeta\big)\Big)r^{n-1}\,dr\,dt\\&+\int_0^\infty \rho_0(r)u_0(r)\zeta(0,r)r^{n-1}\,dr=0
\eeqas
as stated in Definition \ref{def:sphsymsoln}.

Finally, we check the conditions on the test function $\zeta$ at the origin, $r=0$. One easily sees that
$$\zeta(t,0)=\int_{|\mathbf{y}|=1}y_j\varphi(t,0)\,dS_\mathbf{y}=0.$$
However, the radial derivative,
$$\zeta_r(t,0)=\int_{|\mathbf{y}|=1}y_j\nabla\varphi(t,0)\cdot\mathbf{y}\,dS_\mathbf{y}$$
may not be zero. For example, taking $\varphi(t,\mathbf{x})=x_j\chi(\mathbf{x})\psi(t)$ for some cut-off functions $\chi\in C_c^\infty(B_2(0))$ such that $\chi(\mathbf{x})=1$ on $B_1(0)$ and such that $\psi\in C_c^\infty([0,\infty))$, with $\psi=1$ on $[0,1]$, we obtain
$$\zeta_r(1,0)=\int_{|\mathbf{y}|=1}y_j\mathbf{e}_j\cdot\mathbf{y}\,dS_\mathbf{y}=\int_{|\mathbf{y}|=1}y_j^2\,dS_\mathbf{y}>0.$$


\bibliographystyle{amsplain}


\end{document}